\documentclass{article}
\usepackage[utf8]{inputenc}
\usepackage{float}
  \usepackage{amsmath}
  \usepackage{amsthm}
 \usepackage{amssymb}
\usepackage{amsfonts}
\usepackage{graphics}
\usepackage{hyperref}
 \usepackage[english]{babel}
 \newlength\tindent
\newtheorem{theorem}{Theorem}

\newtheorem{lemma}{Lemma}

\newcommand{\mlegendre}[2]{\left(\frac{#1}{#2}\right)}
\begin{document}
\title{Large gaps between sums of two squareful numbers}
\author{Alexander Kalmynin \footnote{National Research University Higher School of Economics, Moscow, Russia} \footnote{Steklov Mathematical Institute of Russian Academy of Sciences, Moscow, Russia} \\ email: \href{mailto:alkalb1995cd@mail.ru}{alkalb1995cd@mail.ru} 
\and Sergei Konyagin \footnotemark[\value{footnote}] \\ email: \href{mailto:konyagin23@gmail.com}{konyagin23@gmail.com}}
\date{}
\maketitle
\begin{abstract}
Let $M(x)$ be the length of the largest subinterval of $[1,x]$ which does not contain any sums of two squareful numbers. We prove a lower bound
\[
M(x)\gg \frac{\ln x}{(\ln\ln x)^2}
\]
for all $x\geq 3$. The proof relies on properties of random subsets of the prime numbers
\end{abstract}

\section{Introduction}

A non-negative integer $n$ is called squareful (or powerful) if either $n=0$ or in its factorization $n=p_1^{\alpha_1}\ldots p_s^{\alpha_s}$ we have $\alpha_i\geq 2$ for all $i$. Equivalently, $n$ is squareful if and only if it can be represented as $n=a^2b^3$ for some $a,b \in \mathbb Z_{\geq 0}$. One can show (see \cite{Gol}) that the number $k(x)$ of squareful numbers $n$ below $x$ satisfies
\[
k(x)=\frac{\zeta(3/2)}{\zeta(3)}\sqrt{x}+O(x^{1/3}),
\]
which implies that squares form a subset of squareful numbers of positive relative density. However, it turns out that sumsets of these two sets are very different. Namely, if $\mathcal S$ is the set of all sums of two squares, then by theorem of Landau \cite{Lan} we have
\[
|\mathcal S\cap [1,x]|\sim \frac{Kx}{\sqrt{\ln x}},
\]
where $K\approx 0.764$ is the Landau-Ramanujan constant. Now, let us denote by $\mathcal V$ the set of all numbers which are sums of two squareful numbers. V. Blomer \cite{Blo} proved the following estimate
\begin{theorem}
For $x\to +\infty$ we have
\[
V(x):=|\mathcal V\cap [1,x]|=\frac{x}{(\ln x)^{\alpha+o(1)}},
\]
where $\alpha=1-\frac{1}{\sqrt[3]{2}}\approx 0.206299$.
\end{theorem}

Theorem 1 was further improved by V. Blomer and A. Granville \cite{BlGr}, they proved that $V(x)=\frac{x(\ln\ln x)^{O(1)}}{(\ln x)^{\alpha}}$.

In this work, we study the distribution of gaps between elements of $\mathcal V$. Namely, let $x$ be a large positive number, then we define $M(x)$ as the length of the largest subinterval of $[1,x]$ without elements of $\mathcal V$. Theorem 1 implies that $M(x)\gg (\ln x)^{\alpha+o(1)}$, since the largest gap is at least as large as the average gap. Here we prove the following result
\begin{theorem}
For $x\geq 3$ we have
\[
M(x)\gg \frac{\ln x}{(\ln\ln x)^2}.
\]
\end{theorem}

The proof of Theorem 2 is in some ways similar to proofs of lower bounds for gaps between sums of two squares (see, for example, \cite{Erd,Ric,DEKKM}), but also different in several crucial aspects. For example, unlike $\mathcal S$, the set $\mathcal V$ does not avoid any residue classes, so we need to count numbers of solutions of certain congruences throughout the proof.

\section{Proof of Theorem 2}

To prove our result, we will use two auxiliary statements. Let us start with the more 
 elementary one.
\begin{lemma}
Let $P$ be a squarefree number, $A,B$ --- integers with $(AB,P)=1$. Then for any $u$ the number of solutions of the congruence
\[
Ax^2+By^2\equiv u \pmod P
\]
does not exceed $2^L\sigma_1(P)$, where $\sigma_1(P)=\sum\limits_{d\mid P}d\ll P\ln\ln P$ and $L$ is the number of prime factors of $(u,P)$.
\end{lemma}
\begin{proof}
Let $p$ be a prime factor of $P$. Let us prove that the number $S_p$ of solutions of the congruence
\[
Ax^2+By^2 \equiv u \pmod p
\] is at most $p+1$ if $p\nmid u$ and at most $2(p+1)$ if $p\mid u$. For $p=2$ it is clear, so we consider the case of odd prime $p$. For any $a \in \mathbb Z\slash p\mathbb Z$ the number of solutions of the congruence $Ax^2\equiv a \pmod p$ is equal to
\[
1+\mlegendre{Aa}{p}.
\]
From this we obtain the following formula for $S_p$:
\[
S_p=\sum_{a+b=u}\left(1+\mlegendre{Aa}{p}\right)\left(1+\mlegendre{Bb}{p}\right)=
\]
\[
=\sum_{a \mod p}\left(1+\mlegendre{Aa}{p}\right)\left(1+\mlegendre{B(u-a)}{p}\right)=
\]
\[\sum_{a \mod p}\left(1+\mlegendre{Aa}{p}+\mlegendre{B(u-a)}{p}+\mlegendre{ABa(u-a)}{p}\right).
\]
The first three summands in the resulting sum are equal to $p$, $0$ and $0$ correspondingly. Therefore,
\[
S_p=p+\mlegendre{AB}{p}\sum_{a \mod p}\mlegendre{a(u-a)}{p}=p\pm\sideset{}{^*}\sum_{a\mod p} \mlegendre{ua^{-1}-1}{p}.
\]
Here the sum with the symbol $*$ means that we sum over $a\not\equiv 0\pmod p$. Clearly, if $u\not\equiv 0\pmod p$ this sum differs from the complete sum by $\mlegendre{-1}{p}$, hence
\[
S_p=p+\mlegendre{-AB}{p}\leq p+1.
\]
Otherwise, it is equal to $(p-1)\mlegendre{-1}{p}$. Therefore, if $u\not\equiv 0\pmod p$ we have $S_p\leq p+1$ and for $u\equiv 0\pmod p$ we get $S_p<2(p+1)$.
By Chinese Remainder Theorem, the number of solutions of the congruence $Ax^2+By^2\equiv u\pmod P$ is equal to
\[
\prod_{p\mid P}S_p\leq 2^L\prod_{p\mid P}(p+1)=2^L\sigma_1(P),
\]
which was to be proved.
\end{proof}

The next lemma we need is a version of Linnik's theorem on the least prime in arithmetic progression.
\begin{lemma}
For a Dirichlet character $\chi$ set
\[
\psi(x;\chi)=\sum_{n\leq x}\Lambda(n)\chi(n).
\]
Then there is a real number $A>0$ such that for all $x\geq q^A$ and all primitive real characters $\chi \mod q$ for $x\to +\infty$ the inequality
\[
\psi(2x;\chi)-\psi(x;\chi)\leq 0.02x.
\]
holds.
\end{lemma}
\begin{proof}
By \cite[Prop. 18.5]{IwK}, for $x\geq q^A$ for large enough $A$ we have for any $a$ with $(a,q)=1$
\[
\psi(x;q,a)=\frac{x}{\varphi(q)}\left(1-\chi_1(a)R(x)+\theta ce^{-CA}+O\left(\frac{\ln q}{q}\right)\right).
\]
Here $R(x)=\frac{x^{\beta_1-1}}{\beta_1}$ if $\chi_1$ and $\beta_1$ are the exceptional character and the exceptional zero of corresponding $L$-function if the exceptional character exists and its modulus $q_1$ divides $q$ and $R(x)=0$ otherwise, $c,C>0$ are constants, $|\theta|\leq 1$. A primitive character $\chi_1 \mod q_1$ is called exceptional if there is a real zero $\beta_1$ of $L(s,\chi_1)$ with $1-\beta_1<\frac{c_1}{\ln q_1}$, $c_1>0$ is an absolute constant such that for any $q$ there is at most one exceptional zero $\beta$ with $1-\beta<\frac{c_1}{\ln q}$ corresponding to a character to the modulus $q'\leq q$ (existence of such $c_1$ is a consequence of Page's theorem). Consequently,
\[
\psi(2x;\chi)-\psi(x;\chi)=\sum_{(a,q)=1}\chi(a)(\psi(2x;q,a)-\psi(x;q,a))=
\]
\[
\frac{x}{\varphi(q)}\sum_{(a,q)=1}\chi(a)-\frac{R(2x)-R(x)}{\varphi(q)}\sum_{(a,q)=1}\chi(a)\chi_1(a)+ce^{-CA}\theta_1 x+O\left(\frac{x\ln q}{q}\right),
\]
where $|\theta_1|\leq 1$. Now, the first sum is equal to $0$. Due to orthogonality of characters, the second summand is nonpositive. Choosing $A$ large enough, we prove the desired inequality.

In other words, the possible contribution of exceptional character $\mod q$ to the sum we are interested in turns out to always be nonpositive, which allows us to prove Lemma 2.
\end{proof}

We now proceed to the proof of the main theorem.

\begin{proof}[Proof of Theorem 2:]
Let us take a large natural $M$ and construct a positive integer $u$ with $u\leq \exp(CM(\ln M)^2)$ for some constant $C>0$ such that the numbers $u+1,u+2,\ldots,u+M$ do not lie in $\mathcal V$. Set $N=M^D$, where $D=\max(40A+1,100)$ and $A$ is a constant from Lemma 2, $K=[22\ln M]$. Let $\mathcal P$ be the set of all primes between $N$ and $2N$. The number of elements in this set is $\pi(2N)-\pi(N)\sim \frac{N}{\ln N}$. For all $j\leq M$ we choose random subsets $P_j\subset \mathcal P$ of size $K$ independently. Equivalently, random variables $P_j$ are jointly independent and for any subset $X\subset \mathcal P$ with $|X|=K$ we have
\[
\mathbb P(P_j=X)={|\mathcal P|\choose K}^{-1}.
\]
Since for every $j$ and every $N<p<2N$ the relation
\[
\mathbb P(p\in P_j)=\frac{{|\mathcal P|-1\choose K-1}}{{|\mathcal P|\choose K}}=\frac{K}{|\mathcal P|}\ll \frac{K\ln N}{N}\ll (\ln M)^2 M^{-D}
\]
holds, at least two of the sets $P_j$ have non-empty intersection with probability $O(M^{3-D})$. Indeed, if $i\neq j$, then
\[
\mathbb P(P_i\cap P_j\neq \emptyset)\leq \sum_{p\in P_i}\mathbb P(p\in P_j)\ll \sum_{p\in P_i}(\ln M)^2 M^{-D}\ll (\ln M)^3 M^{-D},
\]
where the last estimate holds due to independence of $P_i$ and $P_j$. Summing over all pairs $1\leq i<j\leq M$, we get the desired bound.
Next, let $c\leq M^{40}$. Consider the set $R(c)$ of all primes $p\in \mathcal P$ such that $\mlegendre{-c}{p}=-1$. It is easy to see that there is a primitive real character $\chi_c$ to the modulus $q\leq 4|c|$ such that for all $p\in \mathcal P$ we have $\mlegendre{-c}{p}=\chi_c(p)$. The set $R(c)$ has at least $(0.49+o_M(1))|\mathcal P|$ elements, since
\[
|R(c)|=\frac12\sum_{N<p<2N}(1-\chi_c(p))\geq \frac{1}{2\ln N}\sum_{N<p<2N}(1-\chi_c(p))\ln p=\frac{1}{2\ln N}(\psi(2N)-\psi(N))-
\]
\[-\frac{1}{2\ln N}(\psi(2N;\chi_c)-\psi(N;\chi_c))+O(\sqrt{N}),
\]
where the last summand comes from the contribution of $p^k\in (N,2N)$ with $k>1$. By Lemma 2, the second term in the sum above is at most $0.02N$, hence $|R(c)|\geq \frac{(0.98+o_M(1))N}{2\ln N}=(0.49+o_M(1))\frac{N}{\ln N}=(0.49+o_M(1))|\mathcal P|$.
Consequently, the probability of the event that $P_j$ does not intersect $R(c)$ can be estimated as follows
\[
\frac{{|\mathcal P|-|R(c)|\choose K}}{{|\mathcal P|\choose K}}=\frac{(|\mathcal P|-|R(c)|)(|\mathcal P|-|R(c)|-1)\ldots ((|\mathcal P|-|R(c)|-K+1)}{|\mathcal P|(|\mathcal P|-1)\ldots(|\mathcal P|-K+1)}\leq 
\]
\[
\leq \left(\frac{(|\mathcal P|-|R(c)|-K+1)}{(|\mathcal P|-K+1)}\right)^K\ll 0.52^K\ll M^{-14}.
\]
The probability of existence of at least one pair $c\leq M^8$ and $j\leq M$ with non-intersecting $P_j$ and $R(c)$ is $O(M^{-5})$, hence with large probability all $P_j$ are pairwise disjoint and also intersect with every $R(c)$ for $c\leq M^8$. Here we only use ``small'' values of $c$, but we are going to also utilize $c$ with $M^8\leq c\leq M^{40}$ in the following way: for $j\leq M$ let $\xi_j$ be the random variable
\[
\xi_j=\sum_{\substack{M^8\leq ab\leq M^{40}\\ P_j\cap R(ab)=\emptyset}}\frac{1}{a^{3/2}b^{3/2}}
\]
Computing the expectations, we get
\[
\mathbb E\xi_j=\sum_{M^8\leq ab\leq M^{40}}\frac{\mathbb P(P_j\cap R(ab)=\emptyset)}{a^{3/2}b^{3/2}}\ll 0.52^K\frac{\ln M}{M^4}
\]
Then by Markov's inequality the probability that at least one of these variables satisfies $\xi_j\geq 2^{-K}\frac{\ln M}{M^3}$ is at most $O\left(1.04^KM^{-1}\right)=O(M^{-0.13})$, so with large probability we have $\xi_j\leq 2^{-K}\frac{\ln M}{M^3}$ for all $j$.

 Let us choose some realization $P_j^*$ of these random sets which satisfies the above conditions. Construct now a number $u$ which satisfies a congruence
\[
u+j\equiv 0 \pmod p
\]
 for all $j\leq M$ and all $p\in P_j^*$. This number can be chosen to be at most
\[
P=\prod_{j, p\in P_j^*}p\leq (2M^D)^{MK}\leq\exp(CM(\ln M)^2)
\]
for some $C>0$. Consider the set $U$ of all numbers $v\in [0,P^{10}]$ which are congruent to $u$ modulo $P$.

First of all, let us estimate the proportion of $v\in U$ such that $v+j\equiv 0\pmod {p^2}$ for some $j\leq M$ and some $p\in P_j^*$. For each choice of $j$ and $p$ we get a congruence modulo $pP$, hence at most $P^9p^{-1}+1$ solutions in $[0,P^{10}]$, at most $MK(P^9N^{-1}+1)$ in total. The size of the set $U$ is $P^9+O(1)$, hence the proportion of such exceptional $v$ is $O(MK/N)=O(M^{1-40A})$ and we can disregard such $v$ in further estimations.

Now, for most numbers from $U$ none of the numbers $v+j$ is representable as $a^3x^2+b^3y^2$ with positive integers $a,b,x,y$ and $ab\leq M^8$. Indeed, if $v\in U$ and $v+j=a^3x^2+b^3y^2$, then there is a prime $p\in P_j^*$ with $\mlegendre{-c}{p}=-1$, where $c$ is the squarefree part of $ab$. Then $v+j$ must be divisible by $p^2$: if $x$ or $y$ is not divisible by $p$, then
\[
-\frac{a^3}{b^3}\equiv \frac{x^2}{y^2} \pmod p \text{ or }-\frac{b^3}{a^3}\equiv \frac{y^2}{x^2} \pmod p,
\]
which contradicts the identity $\mlegendre{-c}{p}=-1$, therefore $v$ belongs to the set of abovementioned exceptions.

Next, we prove that for most $v$ there is also no $j\leq M$ with $v+j=a^3x^2+b^3y^2$ and $ab\geq M^8$. Fix $j$, $a$ and $b$. Suppose first that $a$ or $b$ is larger than $P^{8/3}$. Without loss of generality, we assume that $a>P^{8/3}$. The number of $x$ with $a^3x^2\leq P^{10}$ is at most $P^5a^{-3/2}$, the number of admissible $y$ is at most $P^5$. Summing over all $a$, we get at most
\[
\sum_{a>P^{8/3}}\frac{P^{10}}{a^{3/2}}\ll P^{10-4/3}=P^{9-1/3}
\]
possible values. Multiplying by the total amount of values of $j$, we see that the proportion of elements $v$ of $U$ for which $v+j$ is representable by $a^3x^2+b^3y^2$ with $\max(a,b)>P^{8/3}$ is at most $O\left(\frac{M}{P^{1/3}}\right)$.

So, it is enough to assume that $a,b\leq P^{8/3}$. If $(ab,P)=1$, then by Lemma 1 the number of solutions $aX^2+bY^2\equiv u+j \pmod P$ does not exceed $O(2^KP\ln\ln P)$ for all $j\leq M$, because by our construction $u+j$ is divisible by exactly $K$ prime factors of $P$ (if it is divisible by some $p\in P_i$ with $i\neq j$, then so is $u+i$ and we get $p\mid i-j$, which is impossible, since $p\geq N$).  Every such solution corresponds to the numbers $x,y$ with $a^3x^2+b^3y^2\leq P^{10}$ and also $ax\equiv X \pmod P$ and $by\equiv Y \pmod P$. Since $x\leq \frac{P^5}{a^{3/2}}$ and $y\leq \frac{P^5}{b^{3/2}}$, for fixed $X,Y$ the number of representable numbers is $O\left(\frac{1}{P^2}\frac{P^{10}}{a^{3/2}b^{3/2}}\right)$. Summing over all $X,Y$, we get $O\left(2^K\ln\ln P\frac{P^{9}}{a^{3/2}b^{3/2}}\right)$. For fixed $j$, we can also assume that $P_j^*\cap R(ab)=\emptyset$, since otherwise we must have $v+j\equiv 0\pmod {p^2}$ for all $P_j^*\cap R(ab)$ and these exceptions are already counted above. Summing over all such $a$ and $b$, we get at most
\[
O\left(2^KP^9\ln\ln P\sum_{\substack{P^{16/3}\geq ab\geq M^8 \\ P_j\cap R(ab)=\emptyset}}\frac{1}{a^{3/2}b^{3/2}}\right)
\]
new exceptions. To evaluate the sum, we split the interval of summation into two parts: $M^{40}\geq ab\geq M^8$ and $ab>M^{40}$ and notice that the first sum is equal to $\xi_j$, so we get
\[
\sum_{\substack{P^{16/3}\geq ab\geq M^8 \\ P_j\cap R(ab)=\emptyset}}\frac{1}{a^{3/2}b^{3/2}}\leq \sum_{\substack{M^{40}\geq ab\geq M^8 \\ P_j\cap R(ab)=\emptyset}}\frac{1}{a^{3/2}b^{3/2}}+\sum_{ab>M^{40}}\frac{1}{a^{3/2}b^{3/2}}=\xi_j+O\left(\frac{\ln M}{M^{20}}\right)\ll 
\]
\[
\ll \ln M(2^{-K} M^{-3}+M^{-20}).
\]
Since there are $M$ possible choices for $j$, we get at most
\[
O(M2^K P^9 \ln\ln P(\ln M(2^{-K}M^{-3}+M^{-20})))=O(P^9\ln^2 M(M^{-2}+2^KM^{-20}))=O\left(\frac{P^9\ln^2 M}{M^2}\right)
\]
new exceptions.

Finally, if the product $ab$ is not coprime to $P$, set $(ab,P)=Q$. Same consideration as above, but with congruences $\mod P$ replaced by congruences $\mod \frac{P}{Q}$, shows that for fixed $a,b$ we have $O\left(2^K\frac{P}{Q}\ln\ln P\left(\frac{P}{Q}\right)^{-2}\frac{P^{10}}{a^{3/2}b^{3/2}}\right)$ representable numbers. Summing over all $a,b$ and using $Q\mid ab$, we get
\[
O\left(\frac{2^KQP^{9}\tau(Q)\ln\ln P}{Q^{3/2}}\right)=O\left(\frac{2^K\tau(Q)}{\sqrt{Q}}P^{9}\ln\ln P\right).
\]
Therefore, the total number of exceptional $v$ with such $a,b$ can be estimate as
\[
O\left(2^KP^9\ln\ln P\sum_{Q\mid P, Q>1}\frac{\tau(Q)}{\sqrt{Q}}\right).
\]
Due to multiplicativity, the last sum can be bounded as follows:
\[
\sum_{Q\mid P, Q>1}\frac{\tau(Q)}{\sqrt{Q}}=\prod_{p\mid P}\left(1+\frac{2}{\sqrt{p}}\right)-1\leq \left(1+\frac{2}{\sqrt{N}}\right)^{KM}-1\ll \frac{KM}{\sqrt{N}},
\]
because $P$ has $KM$ prime factors, each at least $N$.

Hence, the number of exceptional $v$ is $O(2^KP^9M^{4-D/2})=O(P^9M^{-30})$.
This means that for most $v\in U$ none of the numbers $v+j$ is of the form $a^3x^2+b^3y^2$, so the interval $[0,P^{10}]$ contains a subinterval of length $M$ without sums of squareful numbers which concludes the proof, because $P^{10}\leq \exp(10CM(\ln M)^2).$
\end{proof}

\section{Acknowledgements}

We would like to thank Igor Shparlinski for bringing our attention to the set $\mathcal V$ and gaps between its elements, Petr Kucheriaviy for correcting misprints and Artyom Radomskii for pointing out a mistake in Lemma 1 in the previous version of the article.
The first author was supported by the Basic Research Program of the National Research University Higher School of Economics. 

\bibliographystyle{amsplain}

\end{document}